\documentclass{article}[12pt]

\usepackage{graphicx,tikz,color}
\usepackage{amsmath,amsfonts,amssymb,amsthm,latexsym}
\usepackage[ruled,vlined,linesnumbered]{algorithm2e}

\usepackage{enumerate}

\usetikzlibrary{matrix,positioning}
\usetikzlibrary{arrows,patterns}

\newtheorem{theorem}{Theorem}

\newtheorem{proposition}[theorem]{Proposition}

\newtheorem{corollary}[theorem]{Corollary}

\newtheorem{lemma}[theorem]{Lemma}

\newtheorem{problem}[theorem]{Problem}

\def\vertex(#1){\put(#1){\circle*{1.8}}}
\def\lab(#1)#2{\put(#1){\makebox(0,0)[c]{#2}}}

\newcommand{\ggrz}{\gamma_{gr}^{Z}}
\newcommand{\ggr}{\gamma_{gr}}
\newcommand{\ggrt}{\gamma_{gr}^{t}}
\newcommand{\ggrl}{\gamma_{gr}^{L}}

\newcommand{\NN}{{\mathbb N}}

\newcommand{\2}{\vspace{0.2cm}}

\begin{document}

\title{Graphs with unique Grundy dominating sets}

\author{
Bo\v stjan Bre\v sar$^{a,b}$ \and
Tanja Dravec$^{a,b}$ }

\maketitle

\begin{center}
$^a$ Faculty of Natural Sciences and Mathematics, University of Maribor, Slovenia\\
{\tt bostjan.bresar@um.si; tanja.dravec@um.si}
\\
\medskip

$^b$ Institute of Mathematics, Physics and Mechanics, Ljubljana, Slovenia\\
 
\medskip

\end{center}

\author{
}

\maketitle

\begin{abstract}
Given a graph $G$ consider a procedure of building a dominating set $D$ in $G$ by adding vertices to $D$ one at a time in such a way that whenever vertex $x$ is added to $D$ there exists a vertex $y\in N_G[x]$ that becomes dominated only after $x$ is added to $D$. 
The maximum cardinality of a set $D$ obtained in the described way is called the Grundy domination number of $G$ and $D$ a Grundy dominating set. While a Grundy dominating set of a connected graph $G$ is not unique unless $G$ is the trivial graph, we consider a natural weaker uniqueness condition, notably that for every two Grundy dominating sets in a graph $G$ there is an automorphism that maps one to the other. We investigate both versions of uniqueness for several concepts of Grundy domination, which appeared in the context of domination games and are also closely related to zero forcing.  For each of the four variations of Grundy domination we characterize the graphs that have only one Grundy dominating set of the given type, and characterize those forests that enjoy the weaker (isomorphism based) condition of uniqueness. The latter characterizations lead to efficient algorithms for recognizing the corresponding classes of forests.

%The concept of zero forcing was introduced in the context of linear algebra, and was further studied by both graph theorists and linear algebraists. It is based on the process of activating vertices of a graph $G$ starting from a set of vertices that are already active, and applying the rule that an active vertex with exactly one non-active neighbor forces that neighbor to become active. A set $S\subset V(G)$ is called a zero forcing set of $G$ if initially only vertices of $S$ are active and the described process enforces all vertices of $G$ to become active. The size of a minimum zero forcing set in $G$ is called the zero forcing number of $G$.  While a minimum zero forcing set can only be unique in edgeless graphs, we consider the weaker uniqueness condition, notably that for every two minimum zero forcing sets in a graph $G$ there is an automorphism that maps one to the other. We characterize the class of trees that enjoy this  condition by using properties of minimum path covers of trees. In addition, we investigate both variations of uniqueness for several concepts of Grundy domination, which first appeared in the context of domination games, yet they are also closely related to zero forcing.  For each of the four variations of Grundy domination we characterize the graphs that have only one Grundy dominating set of the given type, and characterize those forests that enjoy the weaker (isomorphism based) condition of uniqueness. The latter characterizations lead to efficient algorithms for recognizing the corresponding classes of forests. 
\end{abstract}


%%%%%%%%%%%%%%%%%%%%%%%%%%%%%%%%%%%%%%%%%%%%%%%%%%%%%%%%%%%%%%%%%%%%%
%%%%%%%%%%%%%%%%%%%%%%%%%%%%%%%%%%%%%%%%%%%%%%%%%%%%%%%%%%%%%%%%%%%%%
\section{Introduction}
%%%%%%%%%%%%%%%%%%%%%%%%%%%%%%%%%%%%%%%%%%%%%%%%%%%%%%%%%%%%%%%%%%%%%
%%%%%%%%%%%%%%%%%%%%%%%%%%%%%%%%%%%%%%%%%%%%%%%%%%%%%%%%%%%%%%%%%%%%%
Finding an extremal set that attains a given graph invariant is the most basic problem concerning graph invariants. Another basic question is how many extremal sets for a given invariant are there in a graph. For instance, this question was studied recently in relation with the number of minimum total dominating sets~\cite{HR} and the number of minimum dominating sets~\cite{ADMR}. In 1985, Hopkins and Staton~\cite{hs-1985} studied the graphs that have a unique maximum independent set, and called them the unique independence graphs. The study of these graphs was continued by Gunther et al.~\cite{ghr-1993}. Recently Jaume and Molina~\cite{jm-2018} provided an algebraic characterization of unique independence trees, which can be used for efficient recognition of such trees. In another paper of Gunther et al.~\cite{ghmr-1994} the trees that have a unique minimum dominating set were characterized, while Haynes and Henning in~\cite{hahe-2002} characterized the trees with a unique minimum total dominating set. Two recent papers considered graphs, and in particular trees, that have unique maximum (open) packings~\cite{bope-2020+,bks-2019}. In addition, it was proved in~\cite{bks-2019} that the recognition of the graphs with a unique maximum (open) packing is polynomially equivalent to the recognition of the graphs with a unique maximum independent set, and that the complexity of all three problems is not polynomial, unless P=NP~\cite{bks-2019}. In this paper, we extend these investigation to several other graph parameters, namely to the four Grundy domination invariants and the zero forcing number. 

Let $G$ be a graph, and $v\in V(G)$. The {\em open} (respectively {\em closed}) {\em neighborhood} of a vertex $v$ in $G$ is the set $N_G(v)$ (respectively $N_G[v]$) that contains all neighbors of $v$ (respectively, $N_G[v]=N_G(v)\cup\{v\}$). We say that a vertex $v$ {\em dominates} vertices in $N_G[v]$. A set $D\subseteq V(G)$ is a {\em dominating set} of $G$ if every vertex $u\in V(G)$ is dominated by some $v\in D$.
A vertex $v$ {\em totally dominates} vertices in $N_G(v)$. A set $D\subseteq V(G)$ is a {\em total dominating set} if every vertex $u\in V(G)$ is totally dominated by a vertex $v\in D$. The minimum cardinality of a (total) dominating set is the ({\em total}) {\em domination} number of $G$, denoted by $\gamma(G)$ (respectively, $\gamma_t(G)$). 
Now, consider domination as a process of adding vertices to a dominating set of a graph $G$, which results in a sequence of vertices in $G$ such that each vertex $x$ in the sequence dominates a vertex that was not dominated by vertices that precede $x$ in the sequence. A longest such sequence in a graph $G$ is called a {\em Grundy dominating sequence} and its length is the {\em Grundy domination number} of $G$, denoted by $\ggr(G)$. Vertices of a Grundy dominating sequence form a {\em Grundy dominating set} of $G$. A {\em Grundy total dominating sequence (resp.~set)} is defined similarly by requiring that every vertex in a sequence totally dominates a vertex that was not totally dominated by preceding vertices. These concepts  were initiated in~\cite{bgmrr-2014,bhr-2016}, where the initial motivation for their study came from domination games. Additionally, the Grundy domination invariants represent the worst case scenario in a procedure in which a dominating set is built by adding vertices one by one to a dominating set; see the book~\cite{bookDomGame} on domination games, and references therein concerning Grundy domination invariants. 

A small modification of the condition for Grundy dominating sequence, by changing one of the closed neighborhoods to open neighborhoods in the definition, yields the so-called L-Grundy dominating sequences and the Z-Grundy dominating sequences. As it turns out, the Z-Grundy dominating sets are closely related to zero forcing sets~\cite{bbgkkptv-2017}. The zero forcing number of a graph was introduced in~\cite{AIM} the main motivation being that its minimum rank is bounded by the zero forcing number. The concept was studied in a number of subsequent papers; see~\cite{Barioli-2010,davila,edholm,ghh-2021+,hog,l-2019} for a short selection. As proved in~\cite{bbgkkptv-2017}, in any graph $G$ the complement in $V(G)$ of a Z-Grundy dominating set is always a zero forcing set and vice-versa. In particular, the Z-Grundy domination number $\ggrz(G)$ equals $n(G)-Z(G)$ in any graph $G$, where $Z(G)$ stands for the zero forcing number of $G$. Consequently any result concerning Z-Grundy domination immediately yields the corresponding result on zero forcing. Lin further explored the relations between all four types of Grundy domination numbers and the corresponding zero forcing and minimum rank parameters~\cite{l-2019}. 
 
In this paper, we study graphs in which, for each of the four Grundy domination numbers, a corresponding Grundy dominating set is unique. If $G$ is graph in which there is only one Grundy dominating set (respectively, Grundy total dominating set, Z-Grundy dominating set, L-Grundy dominating set), then $G$ is a {\em unique Grundy domination graph} (respectively, {\em unique Grundy total domination graph, unique Z-Grundy domination graph, unique L-Grundy domination graph}). Already for the first instance  one easily finds that there are no non-trivial connected graphs that have a unique Grundy dominating set, and also for other invariants the graphs with such uniqueness property are very special. In many graphs, uniqueness of the corresponding Grundy dominating-type set is violated only through the action of an automorphism, which in a sense blurs the picture, since we do not really distinguish between such two sets. Hence, we extend our investigation with the following definitions. 

If $G$ is a graph such that for every two Grundy dominating sets  $A$ and $B$ there exists an automorphism $\phi:V(G)\rightarrow V(G)$ such that $\phi(A)=B$, then $G$ is an {\em iso-unique Grundy domination graph}. In the same way, considering the corresponding sets up to automorphisms we define {\em iso-unique Grundy total domination graphs, iso-unique Z-Grundy domination graphs} and {\em iso-unique L-Grundy domination graphs}). 
By the above observations, the class of unique zero forcing graphs (meaning of which should be clear) coincides with the class of unique Z-Grundy domination graphs, and the same holds for iso-unique variations of both concepts. Since a graph $G$ is (iso-)unique with respect to any of the four mentioned invariants if and only if every connected component of $G$ is (iso-)unique, we will consider just connected graphs, the results of which can be trivially extended to all graphs.

The paper is organized as follows. Each section is devoted to (iso-)uniqueness of one of the four Grundy domination concepts. As it turns out, the approaches to their investigations are quite different, which reflects their specific properties. In Section~\ref{sec:GD}, we investigate unique Grundy domination graphs and prove that trivial graph is the only connected graph having this property. In addition, we show that complete graphs are the only connected iso-unique Grundy domination graphs. In Section~\ref{sec:ZF}, we concentrate just on a weaker uniqueness condition, as the characterization of unique $Z$-Grundy domination graphs follows from~\cite{Barioli-2010}. We present a characterization of iso-unique zero-forcing trees by using the concept of path cover of a tree; the result also leads to a polynomial (quadratic) algorithm for recognizing iso-unique zero forcing trees. The study of uniqueness with respect to Grundy total domination, in Section~\ref{sec:TGD}, is again more involved. The unique Grundy total domination graphs are characterized by using a characterization from~\cite{bhr-2016} of the graphs whose Grundy total domination number equals their order. Moreover, we characterize the iso-unique Grundy total domination trees, and present a linear-time algorithm for recognition of these trees. Section~\ref{sec:LGD} is concerned with the fourth version of Grundy domination, the so-called L-Grundy domination, where we again characterize all graphs that have a unique extremal set for this invariant. In addition, it turns out that all trees are iso-unique L-Grundy domination graphs.  

We complete this section with some basic definitions that will be used in the paper.
The {\em degree}, $\deg_G(v)$, of a vertex $v$ in a graph $G$ is $|N_G(v)|$. A vertex $v$ with $\deg_G(v)=0$ is an {\em isolated vertex}. A {\em leaf} is a vertex of degree $1$. A vertex adjacent to a leaf is a {\em support vertex}. A support vertex adjacent to at least two leaves is a {\em strong support vertex}. 
A path between vertices $u$ and $v$ is a $u,v$-{\em path}.
The {\em distance}, $d_G(u,v)$, between vertices $u$ and $v$ in a graph $G$ is the length (number of edges) of a shortest $u,v$-path. The {\em eccentricity} of a vertex $v$ in a graph $G$ is ecc$_G(v)=\max\{d_G(v,x):\, x\in V(G)\}$. The {\em center} of $G$ is the set of all vertices in $G$, which have minimum eccentricity. It is well known that the center of a tree is either a vertex or a pair of adjacent vertices. 
(We may omit the indices in the corresponding notions if the graph $G$ is clear from the context.) Let $[n] = \{1\ldots,n\}$, where $n\in \NN$. For a graph $G=(V,E)$ its order is denoted by $n(G)$, i.e.\ $n(G)=|V(G)|$. 
For a sequence $S=(v_1,\ldots,v_k)$ of distinct vertices of a graph $G$, $\widehat{S}$ denotes the set of vertices from $S$. (Note that two different sequences $S$ and $T$ may have $\widehat{S}=\widehat{T}$, while the uniqueness in this paper is considered only with respect to the sets, not the sequences.) The initial segment $(v_1,\dots,v_i)$ of $S$ will be denoted by $S_i$. Given sequences $S=(v_1,\ldots,v_k)$ and $S'=(u_1,\ldots,u_m)$ of vertices in $G$  such that $\widehat S\cap\widehat{S'}=\emptyset$, $S\oplus S'$ is the {\em concatenation} of $S$ and $S'$, that is, $S\oplus S'=(v_1,\ldots,v_k,u_1,\ldots,u_m)$.

%%%%%%%%%%%%%%%%%%%%%%%%%%%%%%%%%
%%%%%%%%%%%%%%%% GRUNDY
%%%%%%%%%%%%%%%%%%%%%%%%%%%%%%%%

\section{Unique Grundy domination graphs}
\label{sec:GD}

In this section, we characterize unique Grundy domination graphs and iso-unique Grundy domination graphs. We mention that since both classes of graphs are very simple, these graphs can be efficiently recognized. We start the section by formal definitions; see e.g.~\cite{bgmrr-2014, Erey,nt-2020}.

A sequence $S=(v_1,\ldots,v_k)$ of distinct vertices of $G$ is a {\em closed neighborhood sequence}, if for each $i\in [k]$:

\begin{equation}
\label{eq:defGrundy}
N[v_i] \setminus \bigcup_{j=1}^{i-1}N[v_j]\not=\emptyset.
\end{equation}

\noindent The minimum length of a closed neighborhood sequence $S$ such that $\widehat{S}$ is a dominating set is the domination number $\gamma(G)$ of a graph $G$. The maximum length of a closed neighborhood sequence in $G$ is the {\em Grundy domination number}, $\gamma_{gr}(G)$ of $G$, and a corresponding set $\widehat{S}$ is a {\em Grundy dominating set}. The corresponding sequence $S$ is a {\em Grundy dominating sequence} of $G$, or $\ggr(G)$-{\em sequence} for short.

Let $S=(v_1,\ldots ,v_k)$ be a closed neighborhood sequence. We say that for each $i\in [k]$ vertex $v_i$ {\em footprints (with respect to $S$)} the vertices in $N[v_i] \setminus \bigcup_{j=1}^{i-1}N[v_j]$, and that $v_i$ is the {\em footprinter (with respect to $S$)} of any $u\in N[v_i] \setminus \bigcup_{j=1}^{i-1}N[v_j]$. 
Alternatively, the footprinter of $x$ is $v_i\in S$, where $i\in [k]$ is the minimum index $j$ such that $x\in N[v_j]$.
Since each vertex has a unique footprinter, the function $f: V(G) \to \widehat{S}$ which maps a vertex to its footprinter is well defined.

\begin{proposition}\label{prp:unique}
If $G$ is a graph and $x$ an arbitrary vertex of $G$, then there exists a Grundy dominating sequence of $G$ that contains $x$. 
\end{proposition} 
\begin{proof}
Let $S=(v_1,\ldots , v_k)$ be an arbitrary $\ggr(G)$-sequence of $G$. Suppose  that $x \notin \widehat{S}$. Since $S$ is a closed neighborhood sequence, each vertex $v_i \in S$ footprints at least one vertex $v_i' \in N[v_i]$. Note that since $v_i'$ is footprinted by $v_i$, vertex $v_i'$ is not adjacent to $v_{\ell}$ for any $\ell \in [i-1]$. Since $\widehat{S}$ is a dominating set of $G$, $x$ has at least one neighbor in $\widehat{S}$. Let $v_j \in S$ be the footprinter of $x$. Then $S'=(v_k',v_{k-1}',\ldots , v_1')$, where $v_j'=x$ is a closed neighborhood sequence of $G$ because for any $i \in [k]$, vertex $v_i'$ footprints $v_i$. Since $S'$ has length $k = \ggr(G)$ and it contains $x$, the proof is complete.  
\end{proof}

\begin{corollary}\label{cor:unique}
If $G$ is a graph, and ${\cal S}=\{S:\, S \textrm{ is a }\ggr(G) \textrm{-sequence of }G\}$, then
$$\bigcup_{S\in {\cal S}}{ \widehat{S}=V(G).}$$
\end{corollary}

\begin{corollary}\label{cor:unique_characterization}
If $G$ is a connected graph of order at least $2$, then $G$ is not a unique Grundy domination graph.
\end{corollary}
\begin{proof}
Let $G$ be a connected graph of order at least 2. Suppose that $G$ is a unique Grundy domination graph. Let $S$ be an arbitrary $\ggr(G)$-sequence. Since $G$ is connected and non-trivial, $\ggr(G) \leq n(G)-1$ and hence there exists $x \in V(G)$ that is not contained in $S$. By Proposition~\ref{prp:unique}, there exists a $\ggr(G)$-sequence $S'$ of $G$ that contains $x$. Since $S \neq S'$, we get a contradiction. 
\end{proof}

\begin{theorem}
\label{thm:ggrAutounique}
A connected graph $G$ is an iso-unique Grundy domination graph if and only if $G$ is a complete graph. 
\end{theorem}
\begin{proof}
Let $G$ be a connected iso-unique Grundy domination graph, and assume that $G$ is not complete. Let $S=(v_1,\ldots,v_k)$ be a Grundy dominating sequence in $G$. Since $G$ is not complete, $k \geq 2$. Let $U'$ be the set of vertices of $G$ footprinted by $v_k$ (with respect to $S$) and let $U = V(G) \setminus U'$. Since $S$ is the  closed neighborhood sequence of maximum length, $U'$ is a clique. (Indeed, if $U'$ had two non-adjacent vertices $a$ and $b$, then $(v_1,\ldots,v_{k-1},a,b)$ would be a closed neighborhood sequence of length $\ggr(G)+1$, a contradiction.)
Since $G$ is connected, there exists $u \in U$ that has a neighbor $u' \in U'$. Moreover, $u$ is adjacent to all vertices from $U'$.  Since $u'$ has no neighbors in $\{v_1,\ldots ,v_{k-1}\}$, we have $u\neq v_i$ for any $i \in [k-1]$. Thus the sequences $S'$ and $S''$ that are obtained from $S$ by replacing $v_k$ by $u$ and $u'$, respectively, are closed neighborhood sequences. Note that $v_k$ can be equal to $u$ or $u'$, and thus it is possible that either $S=S'$ or $S=S''$. Since the set $\widehat{S''}$ induces more connected components then the set $\widehat{S'}$, it is clear that there is no automorphism of $G$ that maps $\widehat{S'}$ to $\widehat{S''}$, a contradiction. For the converse note that $\ggr(K_n)=1$, which together with the structure of $K_n$ yields that $K_n$ is iso-unique Grundy domination graph.
\end{proof}

%%%%%%%%%%%%%%%%%%%%%%%%%%%%%%%%%
%%%%%%%%%%%%%%%% ZERO FORCING
%%%%%%%%%%%%%%%%%%%%%%%%%%%%%%%%

\section{Unique zero forcing graphs}
\label{sec:ZF}

In this section, we consider iso-unique zero forcing graphs. We establish that these graphs are exactly the iso-unique Z-Grundy domination graphs, and characterize all trees in this class of graphs. The characterization (see Theorem~\ref{thm:characterization}) leads to a polynomial-time algorithm for recognizing iso-unique zero forcing trees. 

Zero forcing is a propagation model based on the following {\it activation rule}. If all neighbors of an active vertex $u$ except one neighbor $v$ are active, then $v$ becomes active. 
We say that $u$ forces $v$ and write $u \rightarrow v$. A set $S \subseteq V(G)$ is a {\it zero forcing set} of $G$ if initially only vertices  of $S$ are active and the propagation of activation rules enforces all vertices of $G$ to become active.  The zero forcing number $Z(G)$ is the minimum cardinality of a zero forcing set of $G$; see~\cite{AIM}.
A chronological list of forcings describes an order in which the non-active vertices are forced and by which vertices they are forced. A {\em forcing chain} is a maximal sequence of vertices $v_1, \ldots ,v_k$ such that $v_i \rightarrow v_{i+1}$ for all $i \in [k-1]$. Clearly, a vertex of a zero forcing set $S$ starts a forcing chain and each forcing chain yields an induced path in $G$, since one vertex of a forcing chain can force at most one vertex. Moreover, every vertex of $G$ belongs to exactly one forcing chain.

Next, we present a version of Grundy domination, the Z-Grundy domination, which is closely related to zero forcing. Compared to the condition~\eqref{eq:defGrundy} on Grundy domination, the condition of Z-Grundy domination is stricter in that a vertex in the sequence must footprint a vertex different from itself.   

A sequence $S=(v_1,\ldots,v_k)$ of vertices of a graph $G$ is a {\em a Z-sequence}, if for each $i\in [k]$: 

\begin{equation}
\label{eq:defZGrundy}
N(v_i) \setminus \bigcup_{j=1}^{i-1}N[v_j]\not=\emptyset.
\end{equation}

\noindent Vertices of a Z-sequence form a {\em Z-set}, and the maximum size of a Z-set in $G$ is the {\em Z-Grundy domination number}, $\ggrz(G)$, of $G$. A set $\widehat{S}$ of vertices that belong to a maximum Z-set $S$ is a {\em Z-Grundy dominating set}. The corresponding sequence is a $\ggrz(G)$-{\em sequence} or a {\em Z-Grundy dominating sequence} of $G$. If $S$ is a Z-sequence, we also use terms {\em Z-footprints, Z-footprinter}, meaning of which should be clear. It may be interesting to remark that $G$ may have isolated vertices in which case z-Grundy dominating set is not a dominating set of $G$.

The following result which connects zero forcing and Z-Grundy domination was proved in~\cite[Theorem 2.2]{bbgkkptv-2017}.

\begin{theorem}{\rm \cite{bbgkkptv-2017}}
If $G$ is a graph without isolated vertices, then 
$$\ggrz(G)+Z(G)=|V(G)|\,.$$
Moreover, the complement of a (minimum) zero-forcing set of $G$ is a (maximum) Z-set of $G$ and vice versa. 
\label{thm:ZFGr}
\end{theorem}

By Theorem~\ref{thm:ZFGr}, when investigating graphs with uniqueness property with respect to zero forcing or Z-Grundy domination, we can use any of the two definitions. Unique zero-forcing graphs were considered in~\cite{Barioli-2010}, where it was proved that no connected graph of order greater than $1$ has a unique zero forcing set. This also implies that there is no connected non-trivial graph $G$ that is a unique Z-Grundy domination graph. Thus only the weaker uniqueness condition should be studied further with respect to Z-Grundy domination.

Although $\ggrz(G) < n(G)$ holds for any graph $G$, there are many iso-unique zero forcing graphs already in the class of trees. The simplest examples are trees $T$ obtained from two disjoint stars $K_{1,n_{1}}$, $n_1 \geq 2$, with center $c_1$ and $K_{1,n_2}$, $n_2 \geq 2$, with center $c_2$ by adding the edge $c_1c_2$. See Fig.~\ref{fig:Z}, where such an example is depicted with $n_1=5$ and $n_2=2$. Note that vertices of any path $P_4$ in $T$ form a Z-Grundy dominating set. 

\begin{figure}[ht!]
\begin{center}
\begin{tikzpicture}[scale=0.52,style=thick]
\def\vr{4pt}
%% vertices defined %%
%M

%X
\path (0,0) coordinate (a);
\path (-4,-3) coordinate (b);
\path (-2,-3) coordinate (c);
\path (0,-3) coordinate (d);
\path (2,-3) coordinate (e);
\path (4,-3) coordinate (f);

\path (8,0) coordinate (g);
\path (7,-3) coordinate (h);
\path (9,-3) coordinate (i);

%Y

%% edges %%

\draw (b)--(a)--(c);
\draw (d)--(a)--(e);
\draw (f)--(a)--(g);
\draw (h)--(g)--(i);

%% vertices %%

\draw (a) [fill=black] circle (\vr);
\draw (b) [fill=black] circle (\vr);
\draw (c) [fill=white] circle (\vr);
\draw (d) [fill=white] circle (\vr);
\draw (e) [fill=white] circle (\vr);
\draw (f) [fill=white] circle (\vr);

\draw (g) [fill=black] circle (\vr);
\draw (i) [fill=black] circle (\vr);
\draw (h) [fill=white] circle (\vr);

\end{tikzpicture}
\end{center}
\caption{An iso-unique Z-Grundy domination graph; vertices of a $\ggr^Z$-set are shaded, while white vertices form a minimum zero forcing set.}
\label{fig:Z}
\end{figure}

In the rest of this section, we focus on iso-unique zero forcing trees. We start by presenting some more definitions needed in their study. 

A {\em path cover} of a tree $T$ is a set of vertex disjoint induced paths of $T$ that cover all vertices of $T$. A path cover $\mathcal{P}$ of $T$ is
{\em minimum} if no other path cover of $T$ has fewer paths than $\mathcal{P}$, and the  {\em path cover number} $P(T)$ is the number of paths in a minimum path cover. Let ${\mathcal{P}}=\{Q_1,\ldots , Q_{\ell}\}$ be a path cover of a tree $T$. An edge $e=xy$, where $x \in V(Q_i)$, $y \in V(Q_j)$, $i \neq j$, is a {\em connector edge} for $\mathcal{P}$ and the end-vertices $x$ and $y$ of this edge are {\em connector vertices} of ${\mathcal{P}}$. A connector vertex is  {\em interior} if it is an interior vertex of the path of $\mathcal{P}$ in which it is contained. A path cover $\mathcal{P}$ of $T$ is {\em interior} if every connector vertex in $\mathcal{P}$ is interior. 
A tree $T$ is a {\em generalized star} if it contains at most one vertex of degree more than 2. A {\em pendant generalized star} of a tree $T$ that is not a generalized star is an induced subgraph $K$ of $T$ such that there is exactly one vertex $v$ of $K$ with $\deg_T(v)=k+1 \geq 3$, $k$ connected components of $T-v$ are pendant paths and $K$ is a subgraph of $T$ induced by those $k$ pendant paths and $v$. The vertex $v$ is called the {\it mid} vertex of $K$.
%If $T$ is a generalized star, then $T$ is the only pendant generalized star of $T$. definicija pendant generalized star je le za drevesa, ki niso generalized star. 

A path $P$ with $V(P)=\{x_1,\ldots , x_k\}$ and $E(P)=\{x_ix_{i+1}:\, i \in [k-1]\}$ will be denoted by $P:x_1,\ldots ,x_k$. The path $x_k,x_{k-1},\ldots , x_1$ will be denoted by $P^{-1}$. Furthermore, given two vertices $x,y \in V(G)$ of a path $P$, an edge $yv\in E(G)$ and two vertices $v,u \in V(G)$ of a path $R$, we will denote by $xPy,vRu$ the $x,u$-walk in $G$ that starts in $x$ and follows $P$ until $y$, continues to $v$ and then follows $R$ until $u$.   

It is well known that $Z(T)=P(T)$ holds for any tree $T$; see~\cite{AIM, ghh-2021+}. Moreover, the set obtained from a path cover $\mathcal{P}$ of a tree $T$ by taking one end-vertex of each path $P \in \mathcal{P}$ is a zero forcing set of $T$. Conversely, if $S=\{v_1,\ldots ,v_k\}$ is a minimum zero forcing set of $T$, then there exists a minimum path cover ${\mathcal{P}} = \{Q_1,\ldots , Q_k\}$ of $T$ such that $v_i$ is an end-vertex of $Q_i$ for all $i \in [k]$. Indeed, for any $i \in [ k ]$, $Q_i$ can be the forcing chain that starts in $v_i$.  

The trees having a unique minimum path cover turn out to be important in the investigation of iso-unique zero forcing trees. In~\cite[Corollary 16]{HJ}, Hogben and Johnson characterized such trees as the trees having an interior path cover. 

\begin{proposition}{\rm \cite{HJ}}\label{prp:hogben} A minimum path cover $\cal P$ of a tree $T$ is the unique minimum path cover of $T$ if and only if $\cal P$ is an interior path cover. 
\end{proposition}

We continue with a basic observation about minimum path covers in trees.

\begin{lemma}\label{l:OneIsInterior}
If $\mathcal{P}$ is a minimum path cover of a tree $T$ and $e=xy$ is a connector edge of $\mathcal{P}$, then at least one end-vertex of $e$ is an interior connector vertex. 
\end{lemma}
\begin{proof}
Suppose that both connector vertices $x \in V(Q_i)$ and $y\in V(Q_j)$ are end-vertices of $Q_i$ and $Q_j$, respectively. Let $Q_i:x,x_2,\ldots ,x_{\ell}$, and $Q_j:y,y_2,\ldots , y_k$. Then ${\mathcal{P}}'=({\mathcal{P}} \setminus \{Q_i, Q_j\})\cup \{P\}$, where $P:y_kQ_j^{-1}y,xQ_ix_{\ell}$, is a path cover containing less paths than $\mathcal{P}$, a contradiction. 
\end{proof}

\begin{lemma}\label{l:endVerticesAreLeafs}
If $T$ is an iso-unique zero forcing tree and $\mathcal{P}$ is a minimum path cover of $T$, then both end-vertices of any $P \in {\mathcal{P}}$ are leaves of $T$.
\end{lemma}
\begin{proof}
Suppose first that $P \in \mathcal{P}$ is an $x,y$-path with $\deg_T(x)=1$ and $\deg_T(y) > 1$. Since $\mathcal{P}$ is a minimum path cover, there exist minimum zero forcing sets $S$ and $S'$ such that $S\setminus S'=\{x\}$ and $S'\setminus S = \{y\}$. This is a contradiction, since there is clearly no automorphism of $T$ that maps $S$ to $S'$.

Suppose now that $P \in \mathcal{P}$ is an $x,y$-path with $\deg_T(x) > 1$ and $\deg_T(y) > 1$ (note that $x=y$ is also possible). Since $x$ is not a leaf of $T$, it has a neighbor $x_1 \notin V(P)$. Let $P_1: x_{1,1},x_{1,2},\ldots , x_{1,n_1}$ be the path from $\mathcal{P}$ that contains $x_1$. By Lemma~\ref{l:OneIsInterior}, $x_1=x_{1,j_1}$ for $j_1 \in \{2,\ldots , n_1-1\}$. Now, we will find a sequence of $\ell\ge 1$ paths. If $x_{1,1}$ is a leaf of $T$, then $\ell=1$, otherwise $x_{1,1}$ has a neighbor $x_2 \notin V(P_1)$. Let $P_2: x_{2,1},x_{2,2},\ldots , x_{2,n_2}$ be the path from $\mathcal{P}$ that contains $x_2$. By Lemma~\ref{l:OneIsInterior}, $x_2=x_{2,j_2}$ for $j_2 \in \{2,\ldots , n_2-1\}$. If $x_{2,1}$ is a leaf of $T$, then $\ell=2$,  otherwise we continue with this procedure until the path $P_{\ell}:x_{\ell,1},\ldots , x_{\ell,n_l}$, which is the path containing a neighbor $x_\ell$ of $x_{\ell-1,1} \in V(P_{\ell-1})$ and $\deg_T(x_{\ell,1})=1$. Then ${\mathcal{P}}'=({\mathcal{P}}\setminus \{P, P_1, \ldots , P_\ell\}) \cup \{P',  P_1' , \ldots , P_\ell'\}$, where 
$$P': x_{\ell,1}P_{\ell}x_{\ell},x_{\ell-1,1}P_{\ell-1}x_{\ell-1},x_{\ell-2,1},\ldots , x_{1,1}P_1x_1,xPy,\textrm{ and}$$
$$P_i':x_{i,j_i+1}P_ix_{i,n_i}, \textrm{ for all } i\in [\ell],$$
is clearly a minimum path cover of $T$. Since $\mathcal{P}$ is a minimum path cover, there is a minimum zero forcing set $S$ of $T$ with $$S \cap (V(P) \cup V(P_1) \cup \ldots \cup V(P_\ell))=\{x_{1,n_1},x_{2,n_2},\ldots , x_{\ell,n_\ell},x\}.$$ As ${\mathcal{P}}'$ is also a minimum path cover, $S'=(S\setminus \{x\}) \cup \{x_{\ell,1}\}$ is a zero forcing set. Since $S'$ has more leaves than $S$, there is no automorphism of $T$ that maps $S$ to $S'$, which is a contradiction. Thus, $\deg_T(x)=\deg_T(y)=1$.
\end{proof}

\begin{lemma}\label{l:interiorCover}
Let $T$ be an iso-unique zero forcing tree and let $\mathcal{P}$ be a minimum path cover of $T$. If $e=xy$ is a connector edge for $\mathcal{P}$, then connector vertices $x$ and $y$ are interior or one of them is interior and the other is the only vertex of a path from $\mathcal{P}$. 
\end{lemma}
\begin{proof}
By Lemma~\ref{l:OneIsInterior}, at least one connector vertex, say $x$, is interior. For the purpose of contradiction assume that $y$ is an end-vertex of a path $P' \in \mathcal{P}$ of length at least 2. Then $\deg_T(y) \geq 2$, which is a contradiction with Lemma~\ref{l:endVerticesAreLeafs}. 
\end{proof}

\begin{lemma}\label{l:P3}
Let $T$ be an iso-unique zero forcing tree and let $\mathcal{P}$ be a minimum path cover of $T$. If $P \in \mathcal{P}$ contains only one vertex $x$, then the neighbor of $x$ is the center of the path $R \in \mathcal{P}$ of length 2, that is, $R \cong P_3$.
\end{lemma}
\begin{proof}
By Lemma~\ref{l:endVerticesAreLeafs}, $x$ is a leaf of $T$. Let $y$ be the neighbor of $x$ in $T$ and let $R: y_1,y_2, \ldots ,y_k$ be the path of $\mathcal{P}$ that contains $y$. Lemma~\ref{l:OneIsInterior} implies that $y$ is an interior vertex of $R$ and hence $y=y_i$ for some $i \in \{2,\ldots , k-1\}$. By Lemma~\ref{l:endVerticesAreLeafs}, $y_1 $ and $y_k$ are leafs of $T$. For the purpose of contradiction assume that $|V(R)| \geq 4$. Hence at least one of the subpaths of $R$, the $y_1,y_{i-1}$-subpath or the $y_{i+1},y_k$-subpath, has length at least 2. Without loss of generality assume that $i>2$ (otherwise we change the roles of both parts of $R$). Let $P': x,y_iRy_k$ and $R':y_1Ry_{i-1}$. Then ${\mathcal{P}}'=({\mathcal{P}} \setminus \{P,R\}) \cup \{P',R'\}$ is also a minimum path cover of $T$. Since ${\mathcal{P}}'$ contains the path $R'$ with the end-vertex $y_{i-1}$ that is not a leaf of $T$, we get a contradiction with Lemma~\ref{l:endVerticesAreLeafs}.  
\end{proof}

%Let $f:V(T) \to V(T)$ be an automorphism of $T$ and $C$ a center of $T$. Then $f$ maps $C$ to $C$. Even more if  a center of $T$ is one vertex $c$, then any automorphism on $T$ fixes $c$. If a center of $T$ is a pair $c_1,c_2$ of adjacent vertices then either $f(c_i)=c_i$ for any $i \in \{1,2\}$ or $f(c_i)=c_{3-i}$ for any $i \in \{1,2\}$.  Ne vem, kje rabimo tole s centrom... 

The following result is based on the fact that minimum zero forcing sets in a tree can be obtained from minimum path covers by taking one end-vertex from each of the paths in a path cover of $T$. Given a path cover $\cal P$, exchanging two end-vertices of a path $P\in {\cal P}$ and keeping end-vertices of other paths in $\cal P$ fixed, we get two zero forcing sets that differ only in one vertex. This yields the existence of an automorphism of a tree, which leads to the following result.

\begin{lemma}\label{l:autP}
Let $T$ be an iso-unique zero forcing tree and let $\mathcal{P}$ be a minimum path cover of $T$. Let $P:x_1,\ldots , x_{\ell} \in \mathcal{P}$, and let $p=\frac{\ell+1}{2}$ if $\ell$ is odd, and $p=\frac{\ell}{2}$ if $\ell$ is even. 
\begin{enumerate}[(a)]
\item If $\ell$ is odd, then the connected components of $T-x_{p}$ that contain $x_{p-1}$ and $x_{p+1}$, respectively, are isomorphic. 
\item If $\ell$ is even, then the connected components of $T-x_{p}x_{p+1}$ that contain $x_{p}$ and $x_{p+1}$, respectively, are isomorphic.
\end{enumerate}
Moreover, there is an automorphism of $G$ that maps $x_1$ to $x_{\ell}$ and $x_{\ell}$ to $x_1$.  
\end{lemma}
\begin{proof}
Let $T$ be an iso-unique zero forcing tree and let $\mathcal{P}$ be a minimum path cover of $T$. Assume first that $\ell$ is odd. Let $S$ be a minimum zero forcing set of $T$, and assume without loss of generality that $x_1\in S$ and $x_{\ell}\notin S$. (Note that a minimum zero forcing set of $T$ can be taken by choosing a leaf of every path in $\mathcal{P}$.)  On the other hand, $S'=(S-\{x_1\})\cup\{x_{\ell}\}$ is also a minimum zero forcing set of $T$, and therefore there is an automorphism $\alpha$ of $T$ that maps $S$ to $S'$. Let $C_1=C_1^{(0)}$ be the component of $T-x_{p}$ that contains $x_{p-1}$ and $C_2$ be the component of $T-x_{p}$ that contains $x_{p+1}$. If $\alpha$ fixes $x_{p-1}$, then $|S\cap V(C_1)|=|S'\cap V(C_1)|$, which is a contradiction, since $S\cap V(C_1)=(S'\cap V(C_1))\cup\{x_1\}$. Therefore, let $\alpha(x_{p-1})=x_{p-1}^{(1)}$, where $x_{p-1}^{(1)}$ is a neighbor of $\alpha(x_p).$ Let $x_p^{(1)}=\alpha(x_p)$ and note that $x_p^{(1)}=x_p$ is not excluded. Let $C_1^{(1)}=\alpha(C_1)$. Note that $C_1^{(1)}$ is the component of $T-x_{p}^{(1)}$ that contains $x_{p-1}^{(1)}$ and is isomorphic to $C_1$. Furthermore, $\alpha(x_1)=x_1^{(1)}$, which is a leaf in $C_1^{(1)}$. Since $S-\{x_1,x_{\ell}\}=S'-\{x_1,x_{\ell}\}$, we infer that $x_1^{(1)}\in S\cap S'$, unless $C_1^{(1)}=C_2$. Now, using the same arguments we infer that $\alpha$ maps $C_1^{(1)}$ to  $C_1^{(2)}$, which is isomorphic to $C_1^{(1)}$. In particular, $\alpha(x_1^{(1)})=x_1^{(2)}$, which is a leaf in $C_1^{(2)}$. Since $S-\{x_1,x_{\ell}\}=S'-\{x_1,x_{\ell}\}$, we infer that $x_1^{(2)}\in S\cap S'$, unless $C_1^{(2)}=C_2$. By the same reasoning, for any integer $i\geq 0$, $\alpha(C_1^{(i)})=C_1^{(i+1)}$ and $\alpha(x_1^{(i)})=x_1^{(i+1)} \in S \cap S'$, unless $C_1^{(i+1)}=C_2$. Since $S \cap S'$ is finite, there exists $i \geq 0$ such that $C_1^{(i+1)}=C_2$. Hence, $\alpha^{i+1}(C_1)=C_2$ and thus $C_1$ is isomorphic to $C_2$.  

The case (b), when $\ell$ is even can be proved in a similar way.   
\end{proof}

In the above lemmas we presented several necessary conditions for iso-unique zero forcing trees. In the next result we prove that all those conditions together yield a  sufficient condition for $T$ being an iso-unique zero forcing tree.

\begin{theorem}\label{thm:characterization}
A tree $T$ is an iso-unique zero forcing tree if and only if the following conditions are satisfied for every minimum path cover $\cal P$ of $T$:
\begin{enumerate}[(i)]
\item Both end-vertices of a path $P\in {\mathcal{P}}$ are leaves of $T$.
\item If $e$ is a connector edge of $\cal P$, then the connector vertices of $e$ are either both interior or one of them is interior vertex of a path in $\mathcal{P}$ isomorphic to $P_3$ and the other is the only vertex of a path in $\mathcal{P}$. 
\item For every path $P:x_1,\ldots , x_{\ell}$ of $\mathcal{P}$ with
$p=\frac{\ell+1}{2}$ if $\ell$ is odd, and $p=\frac{\ell}{2}$ if $\ell$ is even, the following holds. 
\begin{enumerate}
\item If $\ell$ is odd, then the connected components of $T-x_{p}$ that contain $x_{p-1}$ and $x_{p+1}$ are isomorphic. 
\item If $\ell$ is even, then the connected components of $T-x_{p}x_{p+1}$ that contain $x_{p}$ and $x_{p+1}$, respectively, are isomorphic.
\end{enumerate}
Moreover, there is an automorphism of $G$ that maps $x_1$ to $x_{\ell}$  and $x_{\ell}$ to $x_1$.
\end{enumerate}
\end{theorem}
\begin{proof}
By Lemmas~\ref{l:endVerticesAreLeafs},~\ref{l:interiorCover},~\ref{l:P3} and~\ref{l:autP}, the conditions (i),~(ii), and~(iii) are necessary for a tree $T$ to be an iso-unique zero forcing graph. So, let $T$ be a tree such that for any minimum path cover $\cal P$ of $T$ conditions (i),~(ii), and~(iii) are satisfied.

Let $S$ be a minimum zero forcing of $T$. For every $x\in S$ there is a forcing chain (a path) that starts in $x$.
 %there is a sequence of vertices that become blue as a result of the color-change rule, which starts at vertex $x$, and these vertices lie along a path that has $x$ as its end-vertex. 
Hence, every minimum zero forcing set $S$ yields a path cover $\cal P$ of $T$ such that each vertex of $S$ is an end-vertex of (a unique) path in $\cal P$. Since $Z(T)=P(T)$, $\mathcal{P}$ must be a minimum path cover of $T$. On the other hand, it is also known and easy to see that every minimum zero forcing $S$ of $T$ consists only of end-vertices of the paths from the corresponding path cover $\cal P$. 
By (i), all end-vertices of paths in $\cal P$, and thus all vertices in $S$, are leaves of $T$. Clearly, two different zero forcings $S$ and $S'$ may yield the same path cover $\cal P$ of $T$; however, since they are different, some of the end-vertices of the paths in $\cal P$ are different in $S$ and $S'$.  

Let $S$ and $S'$ be arbitrary minimum zero forcing sets of $T$, and first assume that both $S$ and $S'$ yield the same path cover $\mathcal{P}$; that is, each of $S$ and $S'$ consists of end-vertices of paths from $\mathcal{P}$, one end-vertex of each such path.

{\bf Case 1}: $|S \setminus S'|=1$. Hence there exists a path $P:x_1,\ldots ,x_k$ in $\mathcal{P}$ such that $x_1 \in S$ and $x_k \in S'$. By (iii), there exists an automorphism that exchanges $x_1$ to $x_k$ and consequently maps $S$ to $S'$.

{\bf Case 2}: $|S \setminus S'|=\ell > 1$. Let $Q_1,\ldots , Q_{\ell}$ be the paths of $\mathcal{P}$ for which $S \cup S'$ contains both of their end-vertices. First, let $S_1$ be the set of vertices obtained from $S_0=S$ by replacing the end-vertex of $Q_1$ that is contained in $S$ with the other end-vertex of $Q_1$, that is, with the end-vertex of $Q_1$ that is contained in $S'$. Case 1 implies that there exists an automorphism $f_1:V(T) \to V(T)$ that maps vertices of $S=S_0$ to vertices of $S_1$. We continue with the procedure so that in the $i^{\rm th}$ step, where $i \in \{2,\ldots , \ell\}$, $S_i$ is the set of vertices obtained from $S_{i-1}$ by replacing the end-vertex of $Q_i$ that is contained in $S_{i-1}$ with the other end-vertex of $Q_i$, that is, with the end-vertex of $Q_i$ that is contained in $S'$. Case 1 implies that there exists an automorphism $f_i:V(T) \to V(T)$ that maps vertices of $S_{i-1}$ to vertices of $S_i$. Clearly, $f=f_{\ell} \circ f_{l-1} \circ \ldots \circ f_1$ is an automorphism of $T$ that maps $S$ to $S'$.    

Finally, let $S$ correspond to $\mathcal{P}$ and $S'$ correspond to ${\mathcal{P}}'$, where ${\mathcal{P}}'$ and $\mathcal{P}$ are distinct minimum path covers of $T$. If $T$ is a generalized star, then it follows from (i) and (ii) that $T$ is a star and thus an iso-unique zero forcing tree. Otherwise, let $K$ be an arbitrary pendant generalized star of $T$ with mid vertex $v$. If $\deg_K(v) \geq 3$, then it follows from (i) and (ii) that $K$ is a star. If $K$ is a path, then condition (iii) implies that $v$ is the center of $K$ and clearly $K$ must be an element of any minimum path cover of $T$. In particular, $K$ belongs to $\cal P$ and ${\mathcal{P}}'$. For each pendant generalized star $K$ of $T$ that is a star, we remove from $T$ all except two leaves of $K$ and denote the resulting tree by  $T'$. Here we are assuming that when leaves that belong to either $\cal P$ or ${\cal P}'$ are removed, the resulting path cover instead adopts   leaves that remained in the pendant generalized star. It follows from (ii) that any connector edge of any minimum path cover of $T'$ is interior and thus the minimum path cover is unique by~Proposition~\ref{prp:hogben}. Hence ${\mathcal{P}}$ and ${\mathcal{P}}'$ when restricted to $T'$ coincide. Thus the minimum covers $\mathcal{P}$ and ${\mathcal{P}}'$ can only differ in some of the leaves of a pendant generalized star $K_{1,k}$, where $k-2$ of these leaves are covered by one-vertex paths. In any case, any minimum zero forcing set of $T$ contains exactly $k-1$ vertices (that are leaves of $T$) of any pendant generalized star $K_{1,k}$. Clearly, there is an automorphism that maps $k-1$ leaves to some other $k-1$ leaves, which are all attached to the same support vertex. 
Combining this with the initial case when two zero forcings yielded the same path cover, we deduce that there is an automorphism that maps $S$ to $S'$.
\end{proof}

Fig.~\ref{fig:EX} shows an example of an iso-unique zero forcing tree depicting also its minimum path cover and vertices of a smallest zero forcing.

\begin{figure}[ht!]
\begin{center}
\begin{tikzpicture}[scale=0.5,style=thick]
\def\vr{4pt}
%% vertices defined %%
%M

%X
\path (0,1.8) coordinate (a);
\path (-2,0) coordinate (b);
\path (0,-1.8) coordinate (c);
\path (0,0) coordinate (d);
\path (4,3) coordinate (e);
\path (4,1.5) coordinate (f);
\path (4,0) coordinate (g);
\path (4,-1.5) coordinate (h);
\path (4,-3) coordinate (i);
\path (8,0) coordinate (j);
\path (8,1.8) coordinate (k);
\path (10,1) coordinate (l);
\path (10,-1) coordinate (m);
\path (8,-1.8) coordinate (n);

%elipsas 
\draw[dotted] (0,0) ellipse (0.5cm and 2.3cm);
\draw[dotted] (-2,0) ellipse (0.4cm and 0.4cm);
\draw[dotted] (10,1) ellipse (0.4cm and 0.4cm);
\draw[dotted] (10,-1) ellipse (0.4cm and 0.4cm);
\draw[dotted] (8,0) ellipse (0.5cm and 2.3cm);
\draw[dotted] (4,0) ellipse (0.5cm and 3.5cm);

%% edges %%

\draw (a)--(d)--(c);
\draw (e)--(f)--(g)--(h)--(i);
\draw (b)--(d)--(g)--(j)--(k);
\draw (l)--(j)--(m);
\draw (j)--(n);

%% vertices %%

\draw (a) [fill=black] circle (\vr);
\draw (b) [fill=black] circle (\vr);
\draw (c) [fill=white] circle (\vr);
\draw (d) [fill=white] circle (\vr);
\draw (e) [fill=white] circle (\vr);
\draw (f) [fill=white] circle (\vr);

\draw (g) [fill=white] circle (\vr);
\draw (i) [fill=black] circle (\vr);
\draw (j) [fill=white] circle (\vr);
\draw (h) [fill=white] circle (\vr);
\draw (k) [fill=white] circle (\vr);
\draw (l) [fill=black] circle (\vr);
\draw (m) [fill=black] circle (\vr);
\draw (n) [fill=black] circle (\vr);
\end{tikzpicture}
\end{center}
\caption{An iso-unique zero forcing tree; paths of a minimum path cover are circled by ellipses; vertices of a minimum zero forcing are shaded.}
\label{fig:EX}
\end{figure}

The following corollary of Theorem~\ref{thm:characterization} will be useful in the recognition algorithm for iso-unique zero forcing trees. 
(Note that by Proposition~\ref{prp:hogben}, a minimum path cover is unique if and only if it is an interior path cover.) 

\begin{corollary}
\label{cor:iso-uniquezeroforcing}
If $T$ is an iso-unique zero forcing tree and $T'$ is the tree obtained from $T$ such that for every strong support vertex $v$ of $T$, which is adjacent to more than two leaves, all but two leaves adjacent to $v$ are removed, then $T'$ has the unique minimum path cover. 
\end{corollary}

We are ready to present the announced algorithm for deciding whether a given tree is an iso-unique zero forcing graph. It is based on Theorem~\ref{thm:characterization} and Corollary~\ref{cor:iso-uniquezeroforcing}.

\vspace{5mm}

\noindent {\bf Algorithm Iso-Unique Zero Forcing Tree}
 
\noindent {\bf Input.} A tree $T$.

\noindent {\bf Output.} YES if $T$ is an iso-unique zero forcing tree, NO otherwise.

\begin{itemize}
\item[(1)] Let $T'$ be the tree obtained from $T$ such that for every strong support vertex $v$ of $T$, which is adjacent to more than two leaves, all but two leaves adjacent to $v$ are removed.
\item[] Let $\cal P$ be a minimum path cover of $T'$. If $\mathcal{P}$ is not interior, then {\bf RETURN} NO. 

\item[(2)] Consider $\cal P$ in $T$. 
\item[(3)] For every path $P:x_1,\ldots , x_{\ell}$ of $\mathcal{P}$ with
$p=\frac{\ell+1}{2}$ if $\ell$ is odd, and $p=\frac{\ell}{2}$ if $\ell$ is even, check:
\begin{enumerate}[(a)]
\item If $\ell$ is odd, then the connected components of $T-x_{p}$ that contain $x_{p-1}$ and $x_{p+1}$ are isomorphic. 
\item If $\ell$ is even, then the connected components of $T-x_{p}x_{p+1}$ that contain $x_{p}$ and $x_{p+1}$, respectively, are isomorphic.
\end{enumerate}
Moreover, there is an automorphism of $G$ that maps $x_1$ to $x_{\ell}$  and $x_{\ell}$ to $x_1$.
\end{itemize}
\noindent If true for all paths $P\in {\cal P}$, then {\bf RETURN} YES, otherwise {\bf RETURN} NO.  

\bigskip

The correctness of the algorithm is a direct consequence of Theorem~\ref{thm:characterization} and Corollary~\ref{cor:iso-uniquezeroforcing}. If $T$ is an iso-unique zero forcing tree, then the minimum path cover $\cal P$ of $T'$ is interior by Corollary~\ref{cor:iso-uniquezeroforcing} and thus the algorithm does not stop in step (1). By Theorem~\ref{thm:characterization}, the condition of step (3) is satisfied for every path $P \in \cal P$ and hence the algorithm returns YES. For the converse, if $T$ is not an iso-unique zero forcing tree, then one condition of Theorem~\ref{thm:characterization} is not satisfied. If (i) or (ii) of Theorem~\ref{thm:characterization} does not hold, then the minimum path cover $\cal P$ of $T'$ is not interior and thus the algorithm returns NO. If (iii) does not hold, then step (3) returns NO.

Clearly, one can construct $T'$ from $T$ in linear time. By an algorithm from~\cite{HJ} one can construct a minimum path cover of $T$ in linear time. Indeed, the mentioned algorithm is based on finding a pendant generalized star, and providing a path cover for it, and then continuing the process in the tree from which this pendant generalized star is removed. In addition, checking if the resulting path cover is interior can be done efficiently by going through all connector edges and checking if the end-vertices are interior vertices of their paths. This resolves (1). To consider $\cal P$ in $T$ we only need to add additional one-vertex paths to $\cal P$, which consist of vertices deleted in the previous step. This resolves step (2). For step (3), we apply an algorithm for verifying whether two trees are isomorphic. We can use the classical AHU algorithm for checking tree isomorphism~\cite{AHU}. We slightly modify the algorithm by fixing the two vertices and checking whether an algorithm maps one to the other. This algorithm is linear, and since the number of paths in a path cover is $O(n)$, we derive that entire algorithm performs in $O(n^2)$ time. We summarize these observations in the following result.

\begin{theorem}\label{th:algorithmZF}
Algorithm Iso-Unique Zero Forcing Tree verifies with time complexity $O(n^2)$ whether a given tree is an iso-unique zero forcing tree. 
\end{theorem}

%%%%%%%%%%%%%%%%%%%%%%%%%%%%%%%%%
%%%%%%%%%%%%%%%% TOTAL GRUNDY
%%%%%%%%%%%%%%%%%%%%%%%%%%%%%%%%

\section{Unique Grundy total domination graphs}
\label{sec:TGD}

In this section, we consider graphs in which Grundy total dominating sequences are unique, or iso-unique respectively. As it turns out (see Corollary~\ref{cor:tunique_characterization}) uniqueness property holds precisely in the graphs $G$ in which $V(G)$ is a Grundy total dominating set. It is again more difficult to determine which graphs admit the iso-uniqueness property with respect to Grundy total domination. We manage to characterize trees with this property, which leads to a linear-time algorithm for recognizing iso-unique Grundy total domination trees.

Grundy total domination is in the same relation with total domination as Grundy domination is with (standard) domination. Consequently, the condition in its definition is modified from~\eqref{eq:defGrundy} in such a way that closed neighborhoods are replaced by open neighborhoods. Next, we follow with formal definitions, cf. also~\cite{BGD, bhr-2016, GJKM}. 

A sequence $S=(v_1,\ldots,v_k)$ of vertices of a graph $G$ is an {\em open neighborhood sequence}, if for each $i\in [k]$: 
\begin{equation}
\label{eq:detfGrundy}
N(v_i) \setminus \bigcup_{j=1}^{i-1}N(v_j)\not=\emptyset.
\end{equation}
\noindent Note that if $G$ has no isolated vertices, then the minimum length $k$ of an open neighborhood sequence $S$ such that $\widehat{S}$ is a total dominating set of $G$, is the total domination number $\gamma_t(G)$ of $G$. The maximum length of an open neighborhood sequence in $G$ is the {\em Grundy total domination number}, $\ggrt(G)$, of $G$, and the corresponding set $\widehat{S}$ is a {\em Grundy total dominating set}, while 
$S$ is a {\em Grundy total dominating sequence} of $G$. If $S$ is an open neighborhood sequence, we also use terms {\em t-footprinter, t-footprints}, meaning of which should be clear. As in the case of Z-Grundy domination, $G$ may have isolated vertices in which case Grundy total dominating set is not a total dominating set of $G$.

\begin{proposition}\label{prp:tunique}
If $G$ is a non-empty graph and $x$ an arbitrary non-isolated vertex of $G$, then there exists a Grundy total dominating sequence of $G$ that contains $x$. 
\end{proposition} 
\begin{proof}
Let $S=(v_1,\ldots , v_k)$ be an arbitrary $\ggrt(G)$-sequence and $x \in V(G)$ vertex with at least one neighbor in $G$. Furthermore denote by $I$ the set of isolated vertices of $G$. We may assume that $x \notin \widehat{S}$. Since $S$ is an open neighborhood sequence, each vertex $v_i \in S$ t-footprints at least one vertex $v_i' \in N(v_i)$. Note that since $v_i'$ is footprinted by $v_i$, $v_i'$ is not adjacent to $v_{\ell}$ for any $\ell \in [i-1]$. As $\widehat{S}$ is a total dominating set of $G-I$, $x$ has at least one neighbor in $\widehat{S}$. Let $v_j \in S$ be the t-footprinter of $x$. Then $S'=(v_k',v_{k-1}',\ldots , v_1')$, where $v_j'=x$, is a $\ggrt(G)$-sequence that contains $x$. 
\end{proof}

\begin{corollary}\label{cor:tunique}
If $G$ is a graph, $I$ the set of isolated vertices of $G$, and ${\cal S}=\{S:\, S$  is a $\ggrt(G) \textrm{-sequence of }G\}$, then
$$\bigcup_{S\in {\cal S}}{ \widehat{S}=V(G)\setminus I.}$$
\end{corollary}

\begin{corollary}\label{cor:tunique_characterization}
A connected graph $G$ is a unique Grundy total domination graph if and only if $\ggrt(G)=n(G)$.
\end{corollary}
\begin{proof}
If $G$ is a connected graph with $\ggrt(G)=n(G)$, then the only Grundy total dominating set of $G$ is the set $V(G)$. Thus $G$ is a unique Grundy total domination graph. 

For the converse, let $G$ be a connected non-trivial unique Grundy total domination graph. Hence $G$ has no isolated vertices. Let $S$ be an arbitrary $\ggrt(G)$-sequence.  If $\ggrt(G) \leq n(G)-1$, then there exists $x \in V(G) $ that is not contained in $S$. By Proposition~\ref{prp:tunique}, there exists a $\ggrt(G)$-sequence $S'$ that contains $x$. Since $S \neq S'$, we get a contradiction, which implies $\ggrt(G)=n(G)$.
\end{proof}

A characterization of graphs $G$ having $\ggrt(G)=n(G)$ was proved in the seminal paper on Grundy total domination~\cite{bhr-2016}. 

\begin{theorem}
{\rm \cite[Theorem 4.2]{bhr-2016}}
\label{thm:BHR}
If $G$ is a graph with no isolated vertices, then $\ggrt(G)=n(G)$ if and only if there exists an integer $k$ such that
$n(G)=2k$, and the vertices of $G$ can be labeled $x_1,\ldots,x_k,y_1,\ldots,y_k$ in such a way that \2 \\
\indent $\bullet$ $x_i$ is adjacent to $y_i$ for each $i$, \\
\indent $\bullet$ $\{x_1,\ldots,x_k\}$ is an independent set, and
\\
\indent $\bullet$ $y_j$ is adjacent to $x_i$ implies $i \ge j$.
\end{theorem}

Note that Theorem~\ref{thm:BHR} restricted to trees $T$ simplifies to $\ggrt(T)=n$ if and only if $T$ has a perfect matching. 

%%%%% ISO-
%%%%%%%%%%%%%%%

Next, we consider iso-unique Grundy total domination graph. There are several families of such graphs. In particular, this includes the graphs $G$ with $\ggrt(G)=n(G)$, but we can also extend this family by using the following observation. Vertices $u$ and $v$ in a graph $G$ are {\em open twins} if $N_G(u)=N_G(v)$; also, a vertex $u$ is an {\em open twin} if there exists another vertex $v$ such $u$ and $v$ are open twins. It is easy to see that $\ggrt(G)=\ggrt(G-u)$ if $u$ is an open twin in $G$; see~\cite[Proposition 3.6]{BKN}.

\begin{proposition}
\label{prp:ggrt-twins}
Let $u$ be an open twin in a graph $G$.  If $G$ is an iso-unique Grundy total domination graph, then $G-u$ is also an iso-unique Grundy total domination graph. In addition, if $G$ is a tree, then $G$ is an iso-unique Grundy total domination graph if and only if $G-u$ is an iso-unique Grundy total domination graph. 
\end{proposition}
\begin{proof}

Let $u$ and $v$ be open twins in a graph $G$. 
It is easy to see that at most one of these two vertices belongs to an open neighborhood sequence, and also they are both t-footprinted by the same vertex in any such sequence. Now, there is a natural automorphism $\phi_{u\leftrightarrow v}$ that exchanges $u$ and $v$ and fixes all other vertices of $G$.  Note that $\ggrt(G-u)=\ggrt(G)$, and $S$ is a Grundy total dominating sequence in $G$ if and only $S'$ is a Grundy total dominating sequence in $G-u$, where $S'$ is obtained from $S$ by replacing $u$ with $v$ if necessary (or, otherwise, if $u$ is not in $S$, then $S'=S$). By Proposition~\ref{prp:tunique}, $v$ belongs to a $\ggrt(G)$-set, $u$  belongs to a $\ggrt(G)$-set and $v$  belongs to a $\ggrt(G-u)$-set.

Suppose that $G$ is an iso-unique Grundy total domination graph. The family $\cal F$ of $\ggrt(G-u)$-sets is a subfamily of the family ${\cal F}'$ of $\ggrt(G)$-sets, where ${\cal F}'\setminus{\cal F}$ consists of exactly those $\ggrt(G)$-sets that contain $u$.  Since in ${\cal F}'$  every two sets are exchangeable by an automorphism, the same holds in ${\cal F}$, which consists of $\ggrt(G)$-sets that do not contain $u$. Hence $G-u$ is an iso-unique Grundy total domination graph.

For the second statement of the proposition, when $G$ is a tree, we only need to prove the reversed direction. In this case, an open twin $u$ is necessarily a leaf, adjacent to a vertex $w$, and let $U$ be the set of leaves adjacent to $w$ in $G$.  Let $G-u$ be an iso-unique Grundy total domination graph, and let $S$ be an arbitrary $\ggrt(G-u)$-sequence. We claim that $w$ is t-footprinted with respect to $S$ by a leaf $v\in U$. Suppose that $w$ is t-footprinted with respect to $S$ by $z$, which is a not a leaf. Then no vertex from $U$ belongs to $S$. Note that the sequence $S'$ obtained from $S$ by replacing $z$ with $v\in U$ is an open neighborhood sequence in $G-u$, hence $S'$ is a $\ggrt(G-u)$-sequence. This is a contradiction with $G-u$ being an iso-unique Grundy total domination graph, since $S'$ has more leaves than $S$ and so there is no automorphism of $G-u$ that maps $S$ to $S'$. We infer that $w$ is indeed t-footprinted by a leaf $v\in U$ in any Grundy total dominating sequence in $G-u$. Now, we claim that the same holds in $G$. Notably, if there exists a $\ggrt(G)$-sequence in which $w$ is footprinted by $z$, which is not a leaf, then the same sequence is an open neighborhood sequence in $G-u$, and so it is a $\ggrt(G-u)$-sequence, since $\ggrt(G)=\ggrt(G-u)$. This is a contradiction, which implies that $w$ is t-footprinted by a leaf in any $\ggrt(G)$-sequence in $G$. Let $S$ and $T$ be $\ggrt(G)$-sets. If any of them, say $S$, contains $u$, then $\phi_{u\leftrightarrow v}$ maps $S$ to a $\ggrt(G)$-set $\phi_{u\leftrightarrow v}(S)$, which is at the same time a $\ggrt(G-u)$-set. Since there exists an automorphism of $G-u$ that maps $\phi_{u\leftrightarrow v}(S)$ to the $\ggrt(G-u)$-set $T$ (or $\phi_{u\leftrightarrow v}(T)$, if $u$ is contained in $T$), there exists an automorphism of $G$ that maps $S$ to $T$. 
Hence $G$ is an iso-unique Grundy total domination graph.
 \end{proof}

The second statement of Proposition~\ref{prp:ggrt-twins} does not necessarily hold if $G$ contains a cycle. To see this consider the graph $H$, which is obtained from $C_5$ by adding an open twin to any vertex of the cycle. See Figure~\ref{fig:H}, where two copies of $H$ are depicted. Notice that two different $\ggrt(H)$-sets are depicted in these two copies of $H$, where vertices of a $\ggrt(H)$-set in each copy are black. It is thus clear that $H$ is not an iso-unique Grundy total domination graph. However, $H-u$ is isomorphic to $C_5$, and it is known~\cite[Proposition 6.1]{bhr-2016} that $\ggrt(C_n)=n-1$ if $n\ge 3$ is odd, and by symmetry of $C_n$ it follows that every odd cycle is an iso-unique Grundy total domination graph.

\begin{figure}[ht!]
\begin{center}
\begin{tikzpicture}[scale=0.52,style=thick]
\def\vr{4pt}
%% vertices defined %%
%M

%X
\path (-6,0) coordinate (a');
\path (-7.2,2.3) coordinate (b');
\path (-2.8,2.3) coordinate (c');
\path (-4,0) coordinate (d');
\path (-5,4) coordinate (e');
\path (-4.5,1.9) coordinate (c'1);

\draw (5.1,1.8) node [text=black]{$u$};
\draw (-4.9,1.8) node [text=black]{$u$};

%Y

\path (6,0) coordinate (a);
\path (7.2,2.3) coordinate (b);
\path (2.8,2.3) coordinate (c);
\path (4,0) coordinate (d);
\path (5,4) coordinate (e);
\path (5.5,1.9) coordinate (c1);

%% edges %%

\draw (b)--(a)--(d)--(c)--(e)--(b);
\draw (e)--(c1)--(a);

\draw (b')--(a')--(d')--(c')--(e')--(b');
\draw (e')--(c'1)--(d');

%% vertices %%

\draw (a') [fill=black] circle (\vr);
\draw (b') [fill=black] circle (\vr);
\draw (c') [fill=white] circle (\vr);
\draw (d') [fill=black] circle (\vr);
\draw (e') [fill=black] circle (\vr);
\draw (c'1) [fill=white] circle (\vr);

\draw (a) [fill=white] circle (\vr);
\draw (b) [fill=black] circle (\vr);
\draw (c) [fill=black] circle (\vr);
\draw (d) [fill=black] circle (\vr);
\draw (e) [fill=black] circle (\vr);
\draw (c1) [fill=white] circle (\vr);

\end{tikzpicture}
\end{center}
\caption{Graph $H$ with two $\ggrt(H)$-sets marked by black vertices.}
\label{fig:H}
\end{figure}

From Proposition~\ref{prp:ggrt-twins} we derive that when dealing with iso-unique Grundy total domination trees we may restrict our attention to trees with no open twins. In addition, from any such iso-unique Grundy total domination tree $T$ we can build infinite families of examples of iso-unique Grundy total domination trees by attaching an arbitrary number of leaves to support vertices. In the next result, we characterize all iso-unique Grundy total domination trees with no open twins, which thus essentially gives the description of all iso-unique Grundy total domination trees.

\begin{theorem}
\label{thm:ggrt-forest}
If $T$ is a non-trivial tree with no open twins, then $T$ is an iso-unique Grundy total domination graph if and only if $\ggrt(T)=n(T)$.
\end{theorem}
\begin{proof}
We start with an observation about bipartite graphs with partition $V(G)=A+B$. If $S$ is an open neighborhood sequence in $G$, then the two subsequences $S_A$ and $S_B$ of $S$, which are obtained by taking only vertices in $A$ (respectively, $B$) in the same order as they appear in $S$ are clearly both open neighborhood sequences. More importantly, the subsequences are independent of each other, and so $S_A\oplus S_B$ and $S_B\oplus S_A$ are also open neighborhood sequences (with the same length as $S$).

Now, let $T$ be an iso-unique Grundy total domination tree with no open twins, and let $S$ be a Grundy total dominating sequence of $T$. Assume that on the contrary, $\ggrt(T)<n(T)$, and let $v\in V(T)$ be a vertex, which is not in $S$. Given the bipartition $V(T)=A+B$, we may assume that $v\in B$. In addition, by the observation above, we may assume without loss of generality that $S=S_A\oplus S_B$, where $S_A$ (respectively $S_B$) is the subsequence of the vertices in $S$ that belong to $A$ (respectively $B$). Let $S_B=(u_1,\ldots,u_b)$. Let $w$ be the neighbor of $v$, which is t-footprinted as the latest with respect to $S$ among all vertices in $N(v)$. Let $u_j$ be the t-footprinter of $w$. Clearly, $u_j\in B$. Now, it is easy to see that the sequence $S'$ obtained from $S$ by replacing $u_j$ with $v$ is also an open neighborhood sequence. Indeed, the initial segment $S_A\oplus (u_1,\ldots,u_{j-1})$ is the same in both sequences, $v$ footprints $w$, and since $\cup_{i=1}^{j-1}N(u_i)\cup \{v\}\subseteq \cup_{i=1}^{j}N(u_i)$, the remainder of $S'$ is also an open neighborhood sequence. 

Let $C$ be the center of $T$. Note that $C$ consists of either a vertex $c$ or two adjacent vertices $c$ and $c'$, and we consider the following measure for a set $U\subseteq V(T)$: $$m(U)=\sum_{v\in U}{\min\{d(v,c),d(v,c')\}},$$
which represents the sum of distances of vertices in $U$ from the center. 
(In the case when $C=\{c\}$, the above calculation simplifies, but formally, we may let $c'=c$ and use the same formula.) It is clear that for every automorphism of $T$, which maps a subset $U$ onto a subset $U'$, the equality $m(U)=m(U')$ holds. Therefore, since there is an automorphism that maps the $\ggrt(T)$-set $S$ to $S'$, we infer that $v$ and $u_j$ must be at the same distance from the center $C$, while their common neighbor $w$ is closer by $1$ to the center than each of $v$ and $u_j$. 
It is also clear that $v$ and $u_j$ are not leaves, since $T$ has no open twins. 

Now, consider the sequence $S$ again, and let $T_{vw}$ be the (sub)tree of $T$, which coincides with the component of $T-vw$ that contains $v$. Similarly, let $T_{wv}$ be the (sub)tree that coincides with the component of $T-vw$, which contains $w$ (and contains also $u_j$). Let $S_B'$ be the subsequence of $S_B$ of those vertices that belong to $T_{wv}$, and $S_B''$ be the subsequence of $S_B$ of those vertices that belong to $T_{vw}$. Note that $v$ does not belong to $S_B$ (as it does not belong to $S$), hence the subsequences $S_B'$ and $S_B''$ do not affect one another, because the distance between a vertex in one subsequence and a vertex in the other subsequence is at least $4$. We thus infer that $S_A\oplus S_B'\oplus S_B''$ is also an open neighborhood sequence, which we denote by $S_1$. Clearly, $S_1$ is a Grundy total dominating sequence. However, the neighbor of $v$, which is t-footprinted as the latest with respect to $S_1$ among all vertices in $N(v)$, is a neighbor $z$ of $v$, which lies in $T_{vw}$ (so it is not $w$ as in $S$). Let $u_k\in S_B''$ be the vertex that footprints $z$. Clearly, $u_k\in V(T_{vw})$, and $$d(u_k,c)=d(v,c)+2=d(u_j,c)+2.$$ Now, we replace $u_k$ with $v$ in $S_1$ and call the resulting sequence $S_2$. In the same way as earlier we derive that $S_2$ is also a Grundy total dominating sequence in $G$. However, due to the distances from the center of $v$ and of $u_k$, as shown above, we infer that 
$m(S_2)<m(S)$. This implies that there is no automorphism that maps $S$ onto $S_2$, which is a contradiction with $T$ being an iso-unique Grundy total domination graph. Hence, $\ggrt(T)=n(T)$. 

The reverse direction of the statement of the theorem is trivial. 
\end{proof}

By combining Proposition~\ref{prp:ggrt-twins} and Theorem~\ref{thm:ggrt-forest}, we get a characterization of all iso-unique Grundy total domination trees. The result can be best described by the following algorithm for recognition of such trees. 

\bigskip 

\noindent {\bf Algorithm Iso-Unique Grundy Total Domination Tree}
 
\noindent {\bf Input.} A tree $T$. 

\noindent {\bf Output.} YES if $T$ is an iso-unique Grundy total domination tree, NO otherwise.

\begin{itemize}
\item[(1)] Let $T'$ be the tree obtained from $T$ such that for every  strong support vertex $v$ of $T$, all but one leaf adjacent to $v$ are removed. (Note that $T'$ has no open twins.)

\item[(2)] If $T'$ has a perfect matching, then {\bf RETURN} YES, otherwise {\bf RETURN} NO.  
\end{itemize}

\bigskip

Note that a tree with no strong support vertices has no open twins. Hence, the tree $T'$ obtained in step (1) has no open twins. Hence, by Theorem~\ref{thm:ggrt-forest}, $T'$ is an iso-unique Grundy total domination graph if and only if $\ggrt(T')=n(T')$. By Theorem~\ref{thm:BHR} restricted to trees, $\ggrt(T')=n(T)$ if and only if $T'$ has a perfect matching. 
Finally, by Proposition~\ref{prp:ggrt-twins}, $T$ is an iso-unique Grundy total domination tree if and only if $T'$ is an iso-unique Grundy total domination tree. This proves the correctness of the algorithm. Clearly, each of the steps (1) and (2) can be performed in linear time.

\begin{theorem}\label{th:algorithmGT}
Algorithm Iso-Unique Grundy Total domination Tree verifies in linear time whether a given tree is an iso-unique Grundy total domination tree. 
\end{theorem}

%%%%%%%%%%%%%%%%%%%%%%%%%%%%%%%%%
%%%%%%%%%%%%%%%% L-GRUNDY
%%%%%%%%%%%%%%%%%%%%%%%%%%%%%%%%

\section{Unique L-Grundy domination graphs}
\label{sec:LGD}

In this section, we consider the remaining Grundy domination invariant, the L-Grundy domination number. We characterize unique L-Grundy domination graphs, and give some remarks for the iso-unique version. We start with the necessary formal definitions.

A sequence $S=(v_1,\ldots,v_k)$ of distinct vertices of $G$ is an {\em L-sequence}, if for each $i\in [k]$:

\begin{equation}
\label{eq:defLGrundy}
N[v_i] \setminus \bigcup_{j=1}^{i-1}N(v_j)\not=\emptyset.
\end{equation}

\noindent The maximum length of an L-sequence $S$ in $G$ is the {\em L-Grundy domination number}, $\ggrl(G)$, of $G$, and $\widehat{S}$ is an {\em L-Grundy dominating set}. Let $S=(v_1,\ldots ,v_k)$ be an L-sequence. We say that for each $i\in [k]$ vertex $v_i$ {\em L-footprints (with respect to $S$)} the vertices in $N[v_i] \setminus \bigcup_{j=1}^{i-1}N(v_j)$, and that $v_i$ is the {\em L-footprinter (with respect to $S$)} of any $u\in N[v_i] \setminus \bigcup_{j=1}^{i-1}N(v_j)$. Note that a vertex may be L-footprinted twice, once by itself, and later by one of its neighbors; see also~\cite{bbgkkptv-2017,hs}. This fact makes the proof of the next result slightly more involved than the proofs of similar results for other versions of Grundy domination. 

\begin{proposition}\label{prp:Lunique}
If $G$ is a graph and $x$ an arbitrary vertex of $G$, then there exists a $\ggrl(G)$-sequence that contains $x$. 
\end{proposition}
\begin{proof}
Let $x$ be an arbitrary vertex of $G$ and let $S=(v_1,\ldots ,v_k)$ be a $\ggrl(G)$-sequence. For any $i \in [k]$ we denote by $v_i'$ an arbitrary vertex that is $L$-footprinted by $v_i$. If $k=\ggrl(G)=n(G)$ or if $x \in S$, then the statement follows. Thus we may assume that $k \leq n(G)-1$ and that $x \notin S$. Since $\widehat{S}$ is a (total) dominating set, $x$ has at least one neighbor in $\widehat{S}$. Let $v_i \in S$ be the L-footprinter of $x$, with respect to $S$. Since $x \notin \widehat{S}$, $v_i \neq x$. If $N(x) \subseteq N(v_1) \cup \ldots \cup N(v_i)$, then $S'=(v_1,\ldots , v_{i-1},x,v_i,\ldots , v_k)$ is an L-sequence, as $x$ L-footprints itself, $v_i$ L-footprints $x$, and for any $j > i$, $v_j$ $L$-footprints $v_j'$. Since $|S'|=k+1 > \ggrl(G)$ we get a contradiction. Thus $N(x) \nsubseteq N(v_1) \cup \ldots \cup N(v_i)$. Let $v_j \in \{v_{i+1},v_{i+2},\ldots , v_k\}$ be the last vertex from $S$ that L-footprints a vertex from $N(x)$. Let $a \in N(x)$ be a vertex L-footprinted by $v_j$. First note that $a \neq v_j$. Indeed, if $a=v_j$, then $v_j$ footprints itself, and thus it is also footprinted by $v_{\ell}$ for $\ell > j$, which contradicts the choice of $j$. As $a \neq v_j$, $S'=(v_1,\ldots , v_{j-1},x,v_{j+1},\ldots , v_k)$ is an L-sequence ($x$ L-footprints $a$ and for any $\ell > j$, $v_{\ell}$ L-footprints $v_{\ell}'$). Since $x \in S'$ and $|S'|=\ggrl(G)$, the proof is complete.      
\end{proof}

\begin{corollary}\label{cor:Lunique}
If $G$ is a graph, and ${\cal S}=\{S:\, S \textrm{ is a }\ggrl(G) \textrm{-sequence of }G\}$, then
$$\bigcup_{S\in {\cal S}}{ \widehat{S}=V(G).}$$
\end{corollary}

\begin{corollary}\label{cor:Lunique_characterization}
A graph $G$ is a unique L-Grundy domination graph if and only if $\ggrl(G)=n(G)$.
\end{corollary}

There are several graph families that enjoy $\ggrl(G)=n(G)$.  In particular, every forest $T$ with no isolated vertices enjoys $\ggrl(T)=n(T)$; see~\cite[Theorem 5.1]{bghk-2020}. A complete characterization of such graphs is yet to be found.

It is clear that a graph $G$ with $\ggrl(G)=n(G)$ is an iso-unique L-Grundy domination graph and hence all forests are iso-unique L-Grundy domination graphs. There are also (many) iso-unique L-Grundy domination graphs with $\ggrl(G) \leq n(G)-1$. Some simple examples are complete graphs of order at least 3, cycles, and complete bipartite graphs $K_{m,n}$ with $m,n \geq 2$.

%%%%%%%%%%%%%%
%%% DODATNO
%%%%%%%%%%%%%%%%
\section{Concluding remarks}

In this paper, we presented characterizations of graphs that have unique Grundy dominating sets for all four different types of Grundy domination. All characterizations yield very special graphs. In the cases of unique Grundy total domination graphs and unique L-Grundy domination graphs these are the graphs in which the corresponding Grundy dominating set coincides with the vertex set of the graph (minus isolated vertices in the former case). While the former graphs have been characterized (Theorem~\ref{thm:BHR}), the structure of the latter graphs is largely unknown, hence we propose the following

\begin{problem}
Characterize the graphs $G$ with $\ggrl(G)=|V(G)|$. 
\end{problem} 

When uniqueness condition is weakened by involvement of automorphisms,
the situation is completely resolved for Grundy domination; notably the iso-unique Grundy domination graphs are precisely the graphs in which connected components are cliques (Theorem~\ref{thm:ggrAutounique}). In the  other three cases, we could only give characterizations of forests that enjoy the iso-uniqueness condition. The case of iso-unique L-Grundy domination forests is in a sense trivial, since it is known that all forests $T$ enjoy $\ggrl(T)=|V(T)|$, which implies that all forests are iso-unique L-Grundy domination graphs. The cases of Grundy total domination and Z-Grundy domination are much more involved, but we could provide characterizations that yield efficient algorithms for the recognition of these two classes of forests. For three of the four invariants, the recognition of the corresponding forests can be found in linear time or is even trivial, while the algorithm for recognition of iso-unique zero forcing forests  (see Theorem~\ref{th:algorithmZF}) is quadratic. The question remains whether this is optimal or a sub-quadratic algorithm exists. 

Among the three unresolved iso-unique classes of graphs, we feel that iso-unique Grundy total domination graphs could the most accessible ones. By Proposition~\ref{prp:ggrt-twins} we know that removing an open twin from an iso-unique Grundy total domination graph gives a graph in this class. We also know that graphs $G$ with $\ggrt(G)=n(G)$ and odd cycles are in this class, but it would be interesting to find more examples of iso-unique Grundy total domination graphs. The ultimate goal is to resolve the following problem.

\begin{problem}
Characterize the iso-unique Grundy total domination graphs. 
\end{problem}

Another natural open question is whether one can extend the polynomial recognition algorithms from iso-unique Grundy total domination forests, and iso-unique zero forcing forests, respectively, to larger classes of graphs. In addition, since the recognition of iso-unique Grundy domination graphs is polynomial, it would be interesting to see if the same holds for the other three iso-unique classes of graphs. 

\begin{problem}
Is there a polynomial time algorithm to recognize the class of iso-unique Grundy total domination graphs, iso-unique L-Grundy domination graphs, or iso-unique zero forcing graphs, respectively?
\end{problem}
 
If the answer to some of the above questions is negative (which we suspect), one could restrict it to special families of graphs that contain forests. In particular, what happens with the complexity of the above three problems in chordal graphs? 
%%%%%%%%%%%%%%%%%%%%%%%%%%%%%%%%%%%%%%%%%%%%%%%%%%%%%%%%%%%%%%%%%%%%%
%%%%%%%%%%%%%%%%%%%%%%%%%%%%%%%%%%%%%%%%%%%%%%%%%%%%%%%%%%%%%%%%%%%%%

%%%%%%%%%%%%%%%%%%%%%%%%%%%%%%%%%%%%%%%%%%%%%%%%%%%%%%%%%%%%%%%%%%%%%
%%%%%%%%%%%%%%%%%%%%%%%%%%%%%%%%%%%%%%%%%%%%%%%%%%%%%%%%%%%%%%%%%%%%%

\section*{Acknowledgement}

The authors acknowledge the financial support from the Slovenian Research and Innovation Agency (research core funding No.\ P1-0297 and project grants N1-0285 and J1-4008). 

%%%%%%%%%%%%%%%%%%%%%%%%%%%%%%%%%%%%%%%%%%%%%%%%%%%%%%%%%%%%%%%%%%%%%
%%%%%%%%%%%%%%%%%%%%%%%%%%%%%%%%%%%%%%%%%%%%%%%%%%%%%%%%%%%%%%%%%%%%%
%\nocite{*}

%%%%%%%%%%%%%%%%%%%%%%%%%%%%%%%%%%%%%%%%%%%%%%%%%%%%%%%%%%%%%%%%%%

\end{document}